\documentclass[12pt, reqno]{amsart}
\usepackage{ amsmath,amsthm, amscd, amsfonts, amssymb, graphicx, color}
\usepackage[bookmarksnumbered, colorlinks, plainpages]{hyperref}
\newtheorem{thm}{Theorem}[section]
\newtheorem{cor}[thm]{Corollary}
\newtheorem{lem}[thm]{Lemma}

\newtheorem{exam}[thm]{Example}
\numberwithin{equation}{section}

\begin{document}

\title{Rings additively generated by idempotents and nilpotents}

\author{Huanyin Chen}
\author{Marjan Sheibani$^*$}
\email{<sheibani@fgusem.ac.ir>}
\address{
Department of Mathematics\\ Hangzhou Normal University\\ Hang -zhou, China}
\email{<huanyinchen@aliyun.com>}
\address{
Women's University of Semnan (Farzanegan), Semnan, Iran}

\thanks{$^*$Corresponding author: Marjan Sheibani}


\begin{abstract}
A ring $R$ is a strongly 2-nil-clean if every element in $R$ is the sum of two idempotents and a nilpotent that commute. A ring $R$ is feebly clean if every element in $R$ is the sum of two orthogonal idempotents and a unit. In this paper, strongly 2-nil-clean rings are studied with an emphasis on their relations with feebly clean rings. This work shows new interesting connections between strongly 2-nil-clean rings and weakly exchange rings. 
\end{abstract}

\maketitle

\section{Introduction}

Throughout, all rings are associative with an identity. An element $a$ in a ring $R$ is strongly nil-clean provided that every element in $R$ is the sum of an idempotent and a nilpotent that commute (see~\cite{D}). A ring $R$ is a strongly 2-nil-clean if every element in $R$ is the sum of two idempotents and a nilpotent that commute.
As is well known, A ring $R$ is strongly 2-nil-clean if and only if every element in $R$ is the sum of a tripotent and a nilpotent that commute (see~\cite[Theorem 2.8]{CS}).

A ring $R$ is an exchange ring provided that for any $a\in R$ there
exists an idempotent $e\in R$ such that $e\in aR$ and $1-e\in (1-a)R$. A ring $R$ is a weakly exchange ring provided that for any $a\in R$ there
exists an idempotent $e\in R$ such that $e\in aR$ and $1-e\in (1-a)R\bigcup (1+a)R$. Such rings have been studied extensively by many authors (see~\cite{Dan, SP}).
In~\cite[Corollary 2.15]{D}, it was proved that a ring $R$ is strongly nil-clean if and only if $R$ is an exchange ring and a UU ring. Here, a ring $R$ is a UU ring if every unit in $R$ is a unipotent.

A ring $R$ is feebly clean if every element in $R$ is the sum of two orthogonal idempotents and a unit. Commutative feebly clean rings were extensively investigated by \cite{AK}	, motivated by the work on continuous function rings (see~\cite{AK}).  In this paper, strongly 2-nil-clean rings are studied with an emphasis on their relations with feebly clean rings. This work shows new interesting connections between strongly 2-nil-clean rings and weakly exchange rings. The Danchev's problem~\cite{?} was thereby answered.

We use $N(R)$ to denote the set of all nilpotent elements in $R$ and $J(R)$ the Jacobson radical of $R$.
An element $u\in R$ is a unipotent if $1-u\in N(R)$. Two idempotents $e,f\in R$ are orthogonal if $ef=fe=0$. ${\Bbb N}$ stands for the set of all natural numbers.

\section{Feebly Clean Rings}

The aim of this section is to characterize strongly 2-nil-clean rings by means of feeble cleanness. Recall that a ring $R$ is 2-UU if $u^2$ is a unipotent for all $u\in U(R)$. We begin with

\begin{lem} Let $R$ be a feebly clean 2-UU ring. Then $6\in R$ is nilpotent.\end{lem}
\begin{proof} Write $3=e-f+u$ where $e,f$ are orthogonal idempotents and $u$ is a unit.
Set $g=e-f$. Then $g=g^3$. Since $R$ is a 2-UU ring, $u^2=1+w$ for some $w\in N(R)$. Then
$$8(g+u)=24=3^3-3=(g+u)^3-(g+u)=3g^2u+3gu^2+v,$$ where $v=u^3-u=u(u^2-1)=uw$. We note that $uw=u(u^2-1)=(u^2-1)u=wu$. Hence, $v=uw\in N(R)$.
Multiplying both sides by $gu$, we get
$8(g^2u+gu^2)=3(g^2u+gu^2)+t$ for some $t\in N(R)$, and so $5(g^2u+gu^2)=t$. Thus, $2^3\cdot 3\cdot 5=5\cdot 24=
5\cdot (3g^2u+3gu^2+v)=3\cdot 5(g^2u+gu^2)+5v=3t+5v\in N(R)$.
Therefore $2\times 3\times 5\in N(R)$. Write $2^m\cdot 3^m\cdot 5^m=0$. Then $R\cong R/2^mR\times R/3^mR\times R/5^mR$.
Set $R_3=R/5^mR$. Then $5\in N(R_3)$, and so $4=5-1\in U(R_3)$. This implies that $2\in U(R_3)$. By hypothesis, we easily see that $R_3$ is a 2-UU ring. Thus, $2^2\in 1+N(R_3)$; hence, $3\in N(R_3)$.
As $(3^m, 5^m)=1$, we see that $5\in U(R_3)$, a contradiction. Therefore $R\cong R/2^mR$, or $R/3^mR$, or the product of such rings.
This implies that $6\in N(R)$, as asserted.\end{proof}

\begin{lem} Let $R$ be a feebly clean 2-UU ring. Then $J(R)$ is nil.\end{lem}
\begin{proof} By Lemma 2.1, $6\in N(R)$. Say $6^n=0$. Then $R\cong R_1\times R_2$, where $R_1\cong R/2^nR, R_2\cong R/3^nR$.
As $2\in N(R_1)$, We have $a=e+u$ for some $u\in U(R_1)$,  as $R$ is 2-UU ring then so is $R_1$ so $u^2=1+w$ for some $w\in N(R_1)$, also $2\in N(R_1)$ then, $(u-1)(u+1)+2(1-u)\in N(R_1)$, this implies that $(u-1)\in N(R_1)$. We get $a=e+v+1$ for some $v\in N(R_1)$. We deduce that $R_1$ is strongly 2-nil-clean. According to~\cite[Theorem 3.3]{CS},
$J(R_1)$ is nil. Let $x\in J(R_2)$. As $R_2$ is a 2-UU ring, $(1+x)^2=1+w$ for some $w\in N(R_2)$, i.e., $x(x+2)=w$. As $3\in N(R_2)$, we see that $2=3-1\in U(R_2)$ and so $x+2=2(1+2^{-1}x)\in U(R_2)$. By applying $(x+2)^{-1}w=w(x+2){-1}$, we deduce that $x=w(x+2){-1}\in N(R_2)$, and so $J(R_2)$ is nil.

Accordingly, $J(R)$ is nil, hence the result.\end{proof}

\begin{lem} Let
$I$ be an ideal of a feebly clean 2-UU ring $R$. If $R/I$ is a domain and $3\in R$ is nilpotent, then
every unit lifts modulo $I$.\end{lem}
\begin{proof} Take $\overline{a}\in U(R/I)$. Since $R$ is feebly clean, we can find
orthogonal idempotents $e,f\in R$ and $u\in U(R)$ such that $a=e-f+u$. Since $R/I$ is a domain, $\{\overline{e},\overline{f}\}\subseteq \{
\overline{0},\overline{1}\}$ in $R/I$.

As $R$ is a 2-UU ring, $u^2=1+w$ for some $w\in N(R)$. Hence, $\overline{u}^2=\overline{1}$ in $R/I$, and so $\overline{u}=\pm \overline{1}$.

Case I. $e-f\equiv
0~\big(\mbox{mod}~I\big)$. Then $a\equiv
u~\big(\mbox{mod}~I\big)$.

Case II. $e-f\equiv 1~\big(\mbox{mod}~I\big)$. If $\overline{u}=\overline{1}$, then $a-2\in I$ with $2\in U(R)$. If $\overline{u}=-\overline{1}$, $a\in I$, a contradiction.

Case III. $e-f\equiv -1~\big(\mbox{mod}~I\big)$.  If $\overline{u}=\overline{1}$, then $a\in I$, a contradiction.
If $\overline{u}=-\overline{1}$, then $a+2\in I$ with $2\in U(R)$.

Therefore we complete the proof.\end{proof}

Recall that a ring $R$ is clean if every element in $R$ is the sum of an idempotent and a unit (see~\cite{CH}). We have

\begin{lem} A ring $R$ is strongly 2-nil-clean if and only if
\end{lem}
\begin{enumerate}
\item [(1)] {\it $R$ is feebly clean;}
\vspace{-.5mm}
\item [(2)] {\it $R$ is a 2-UU ring.}
\vspace{-.5mm}
\item [(3)] {\it $N(R)$ forms an ideal of $R$.}
\end{enumerate}
\begin{proof} $\Longrightarrow $ $(1) $ is obvious as every strongly 2-nil-clean ring is clean and so feebly clean.\\
$(2)$ Let $u\in U(R)$, as $R$ is strongly 2-nil-clean, in view of \cite[Theorem 2.8]{CS}, there exist $p^3=p\in R$ and $w\in N(R)$ such that $u=p+w, pw=wp$, then $u^2=p^2+v$ for some $v\in U(R)$ so $u^2-v=p^2$ and $p^4=p^2$, this implies that $p^2=1$ and so $u^2=1+w$ is a unipotent.\\
$(3)$ follows from \cite[Theorem 3.6]{CS}. 

$\Longleftarrow$ By Lemma 2.1, $6\in N(R)$. Say $6^n=0$. Then $R\cong R_1\times R_2$, where $R_1\cong R/2^nR, R_2\cong R/3^nR$.
Clearly, $R_i$ is feebly clean, $R_i$ is a 2-UU ring and $N(R_i)$ forms an ideal of $R_i$ for $i=1,2$.

Step 1. Let $a\in R_1$. Then there exists orthogonal idempotents $e,f\in R$ and a unit $u\in R$ such that $a=e-f+u$.
Hence, $a=(e+f)+(2f+u)$. Clearly, $(e+f)^2=e+f$. As $2\in N(R_1)$, we see that $2f+u=(2fu^{-1}+1)u$ is invertible. Thus, $R_1$ is clean. We have $a=e+u$ for some $u\in U(R_1)$,  as $R$ is 2-UU ring then so is $R_1$ so $u^2=1+w$ for some $w\in N(R_1)$, also $2\in N(R_1)$ then, $(u-1)(u+1)+2(1-u)\in N(R_1)$, this implies that $(u-1)\in N(R_1)$. We get $a=e+v+1$ for some $v\in N(R_1)$. We deduce that $R_1$ is strongly 2-nil-clean.

Step 2. Suppose that $\overline{x}^2=\overline{0}$ in $R_2/J(R_2)$. Then $x^2\in J(R_2)$. In view of Lemma 2.2, $J(R_2)$ is nil; hence, $x\in N(R_2)$. By hypothesis, $x\in J(R_2)$. This shows that
$R_2/J(R_2)$ is reduced. In light of~\cite[Theorem 12.7]{L}, it is the subdirect product of domains $S_i$. This, there exists epimorphisms $\varphi_i: R_2/J(R_2)\to S_i$ such that $\bigcap Ker(\varphi_i)=0$.

Since $R_2$ is feebly clean, then so is $R_2/J(R_2)$. As every unit lifts modulo $J(R_2)$, we see that $R_2/J(R_2)$ is a 2-UU ring. Thus,
$R_2/J(R_2)$ is a feebly clean 2-UU ring with $3\in R_2/J(R_2)$ is nilpotent.

As $S_i$ is domain and $S_i\cong R_2/J(R_2)/Ker(\varphi_i)$. It follows by Lemma 2.3 that every unit modulo $Ker(\varphi)$. It follows that $S_i$ is a 2-UU ring. But $S_i$ is a domain, we see that $U(S_i)=\{ -1,1\}$.

Since $S_i$ is a homomorphic image of $R_2/J(R_2)$, we see that $S_i$ is feebly clean. But all idempotents in $S_i$ are $0,1$, and so $S_i=\{-2,-1,0, 1,2\}$. This implies that $S_i$ is commutative.

Since $R_2/J(R_2)$ is the subdirect product of $S_i$, it is isomorphic to the subring of $\Pi S_i$, and so $R_2/J(R_2)$ is commutative.
Thus, $R_2/J(R_2)$ is strongly feebly clean. According to \cite[Lemma 2.2]{CS}, $R_2/J(R_2)$ is strongly 2-nil-clean. In light of \cite[Lemma 3.1]{CS}, $R_2$ is strongly 2-nil-clean, and so the result is proved.\end{proof}

We have accumulated all the information necessary to prove the following.

\begin{thm} A ring $R$ is strongly 2-nil-clean if and only if
\end{thm}
\begin{enumerate}
\item [(1)] {\it $R$ is feebly clean;}
\vspace{-.5mm}
\item [(2)] {\it $J(R)$ is nil;}
\vspace{-.5mm}
\item [(3)] {\it $U(R/J(R))$ has exponent $\leq 2$.}
\end{enumerate}
\begin{proof} $\Longrightarrow $ $(1)$ and $(2)$ are obvious by ~\cite[Lemma 2.2]{CS} and \cite[Theorem 3.3]{CS} . Let $\overline{u}\in U(R/J(R))$. Then $u\in U(R)$.  $u^2=1+w$, where $w\in N(R)\subseteq J(R)$. Thus, $\overline{u}^2=\overline{1}$, as desired.

$\Longleftarrow$ As in the proof of Lemma 2.4, $R\cong R_1\times R_2$ with $2\in N(R_1)$ and $3\in N(R_2)$.
Clearly, each $S_i$ is feebly clean, $J(S_i)$ is nil and $U(S_i/J(S_i))$ has exponent $\leq 2$.

Step 1. Let $u\in U(R_1)$. Then $u^2=1+r$ for some $r\in J(R_1)$, and so $r\in N(R_1)$. Thus, $R_1$ is a 2-UU ring.
As $2\in N(R_1)$, we see that $R_1$ is clean. Thus, $R_1$ is strongly 2-nil-clean as we see in the proof of Lemma 2.4.

Step 2. Let $a\in N(R_2)$. Then $1+a\in U(R_2)$. Hence, $(1+a)^2=1+w$ for some $w\in J(R_2)$.
Hence, $a(a+2)=w$. As $2\in U(R_2)$, we see that $a+2\in U(R_2)$. Therefore $a=w(a+2)^{-1}\in J(R_2)$. Therefore $N(R_2)=J(R_2)$ is an ideal of $R_2$.
Let $u\in U(R_2)$. Then $\overline{u}\in U(R_2/J(R_2))$; hence, $\overline{u}^2=\overline{1}$; whence, $u^2\in 1+J(R_2)\subseteq 1+N(R_2)$.
Thus, $R_2$ is a 2-UU ring. In light of Lemma 2.4, $R_2$ is strongly 2-nil-clean.

Therefore $R$ is strongly 2-nil-clean, as asserted.\end{proof}

\begin{cor} A ring $R$ is strongly nil-clean if and only if
\end{cor}
\begin{enumerate}
\item [(1)] {\it $R$ is feebly clean;}
\vspace{-.5mm}
\item [(2)] {\it $J(R)$ is nil;}
\vspace{-.5mm}
\item [(3)] {\it $U(R)=1+J(R)$.}
\end{enumerate}
\begin{proof} $\Longrightarrow $ This is obvious, by [?????????].

$\Longleftarrow$ Clearly, $U(R/J(R))$ has exponent $\leq 2$. In view of Theorem 2.5, $R$ is strongly 2-nil-clean. As $-1\in 1+J(R)$, we see that $2\in J(R)$ is nil. According to \cite[Theorem 2.11]{CS}, $R$ is strongly nil-clean.\end{proof}

\begin{exam} Let $R={\Bbb Z}_{(2)}\bigcap {\Bbb Z}_{(3)}=\{ \frac{m}{n}~|~(m,n)=1, m,n\in {\Bbb Z}, 2,3\nmid n\}$. Then $R$ is feebly clean and $U(R/J(R))$ has exponent $\leq 2$, but $R$ is not strongly 2-nil-clean.\end{exam}
\begin{proof} In view of~\cite[Example 3.3]{AK}, $R$ is feebly clean. Since $J(R)=2R\bigcap 3R$, $$R/J(R)\cong R/2R\times R/3R\cong {\Bbb Z}_2\times {\Bbb Z}_3;$$ hence,
$U(R/J(R))=\{(1,1),(1,-1)\}$, which has exponent $\leq 2$. But $R$ is not strongly 2-nil-clean, as $J(R)$ is not nil.\end{proof}

\section{Weakly Exchange Properties}

The goal of this section is to characterize strongly 2-nil-clean rings by means of weakly exchange rings. In fact we extend the results in~\cite{D} from exchange rings to weakly exchange rings. An element $a\in R$ is exchange if there exists an idempotent $e\in aR$ such that v$1-e\in (1-a)R$. We have

\begin{lem} Let $R$ be weakly exchange. If $R$ is a 2-UU ring, then $6\in R$ is nilpotent.\end{lem}
\begin{proof} Since $R$ is weakly exchange, then $3$ or $-3$ is exchange. 

Case 1. $3\in R$ is exchange. Then there exists an idempotent $e\in R$ such that $e\in 3R$ and $(1-e)\in (1-3)R$. There exist $a,b\in R$ such that $e=3a$ and $(1-e)=-2b$, where $ae=a$ and $b(1-e)=b$. Now $3=(1-e)+(3-(1-e))$. It is easy to prove that $3-(1-e)$ is a unit with inverse $a-b$.

Case 2. $-3$ is exchange. By the similar argument above, we can find an idempotent $e\in R$ and a unit $v\in R$ such that $-3=e+v$. Hence, $3=-e-v$. 

Accordingly, $3\in R$ is feebly clean. As in the proof of Lemma 2.1, we conclude that $6\in N(R)$.\end{proof}

Recall that a ring $R$ is weakly clean if every element in $R$ is the sum or difference of a nilpotent and an idempotent (see~\cite{DA}). 

\begin{lem} A ring $R$ is strongly 2-nil-clean if and only if
\end{lem}
\begin{enumerate}
\item [(1)] {\it $R$ is weakly clean;}
\vspace{-.5mm}
\item [(2)] {\it $J(R)$ is nil;}
\vspace{-.5mm}
\item [(3)] {\it $U(R/J(R))$ has exponent $\leq 2$.}
\end{enumerate}
\begin{proof} $\Longrightarrow $ As $R$ is strongly 2-nil-clean, by~\cite[Proposition 3.5]{CS} it is strongly clean and then it is weakly clean. $(2), (3)$ follow from Theorem 2.5.

$\Longleftarrow$ Clearly, $R$ is feebly clean and so the result follows from Theorem 2.5. \end{proof}

\begin{thm} A ring $R$ is strongly 2-nil-clean if and only if
\end{thm}
\begin{enumerate}
\item [(1)] {\it $R$ is weakly exchange;}
\vspace{-.5mm}
\item [(2)] {\it $J(R)$ is nil;}
\vspace{-.5mm}
\item [(3)] {\it $U(R/J(R))$ has exponent $\leq 2$.}
\end{enumerate}
\begin{proof} $\Longrightarrow $ In view of~\cite[Proposition 3.5]{CS}, every strongly 2-nil-clean ring ring is clean, and so it is weakly exchange. Thus, this implication is obtained by Theorem 2.5.

$\Longleftarrow$ Let $u\in U(R)$. Then $\overline{u}^2=\overline{1}$ in $R/J(R)$. Hence, $u^2-1\in J(R)\subseteq N(R)$, and so $R$ is a 2-UU ring.
In view of Lemma 4.1, $6\in N(R)$. Write $2^n3^n=0$. Then $R\cong R_1\times R_2$ where $R_1=R/2^nR$ and $R_2=R/3^nR$.
Obviously, each $S_i$ is weakly exchange, $J(S_i)$ is nil and $U(S_i/J(S_i))$ has exponent $\leq 2$.

Step 1. Let $u\in U(R_1)$. Then $u^2=1+r$ for some $r\in J(R_1)$, and so $r\in N(R_1)$. Thus, $R_1$ is a 2-UU ring.

Since $2\in N(R_1)$, we see that $2\in J(R_1)$. In light of ~\cite[Theorem 2.2]{Dan}, $R_1$ is exchange. Therefore $R_1$ is strongly 2-nil-clean, as we see in the proof of Lemma 2.4.

Step 2. Let $a\in N(R_2/J(R_2))$. Since $J(R_2)$ is nil, we see that $a\in N(R_2)$, and so $1+a\in U(R_2)$. Hence, we can find some $w\in J(R_2)$ such that $(1+a)^2=1+w$. This shows that $a(2+a)=w$.
As $3\in N(R_2)$, we see that $2+a\in U(R_2)$, and so $a=(2+a)^{-1}w\in J(R_2)$. Thus, $R_2/J(R_2)$ is reduced, and so it is abelian. Clearly, $R_2/J(R_2)$ is weakly excahange. In light of ~\cite[???]{???}, $R_2/J(R_2)$ is feebly clean.

Obviously, $J(R/J(R_2))=0$ and $U(R_2/J(R_2))$ has exponent $\leq 2$. Applying Theorem 2.5 to $R_2/J(R_2)$, we see that $R_2/J(R_2)$ is strongly 2-nil-clean.
Since $J(R_2)$ is nil, we show that $R_2$ is strongly 2-nil-clean, by~\cite[Lemma 3.1]{CS}.

Therefore $R$ is strongly 2-nil-clean, as asserted.\end{proof}

\begin{cor} A ring $R$ is strongly nil-clean if and only if
\end{cor}
\begin{enumerate}
\item [(1)] {\it $R$ is weakly exchange;}
\vspace{-.5mm}
\item [(2)] {\it $J(R)$ is nil;}
\vspace{-.5mm}
\item [(3)] {\it $R$ is a UU ring.}
\end{enumerate}
\begin{proof} $\Longrightarrow $ This is obvious.

$\Longleftarrow$ in light of Theorem 3.3, $R$ is strongly 2-nil-clean.By the UU property of $R$, according to \cite[Theorem 2.8]{D} $2\in N(R)$. Then by applying \cite[Theorem 2.11]{CS}, $R$ is strongly nil-clean.\end{proof}

A ring $R$ is strongly weakly nil-clean if every element in $R$ is the sum or difference of a nilpotent and an idempotent that commute (see~\cite{CS2}). We now
turn to describe strongly weakly clean rings and thereby answer the Danchev's problem.

\begin{lem} A ring $R$ is strongly weakly nil-clean if and only if\end{lem}
\begin{enumerate}
\item [(1)]{\it $R$ has no homomorphic image ${\Bbb Z}_3\times {\Bbb Z}_3$;}
\item [(2)]{\it $R$ is strongly 2-nil-clean.}
\vspace{-.5mm}
\end{enumerate}\begin{proof} $\Longrightarrow$ If ${\Bbb Z}_3\times {\Bbb Z}_3$ is a homomorphic image of $R$, then it is strongly weakly nil-clean.
But $(1,-1)\in {\Bbb Z}_3\times {\Bbb Z}_3$ is not strongly weakly nil-clean. This gives a contradiction. Thus proving $(1)$.
Let $a\in R$. In view of~\cite[Theorem 2.1]{CS2}, $a\pm a^2\in N(R)$. Hence, $a^2-a^4=(a-a^2)(a+a^2)\in N(R)$. Then $a-a^3\in N(R)$, and thus proving $(2)$ by~\cite[Theorem 2.3]{CS}.

$\Longleftarrow$ In view of~\cite[Lemma 4.1]{CS}, $R/J(R)$ is isomorphic to a Boolean ring, a Yaqub ring, or the product
of such rings. By $(1)$, $R/J(R)$ is isomorphic to a Boolean ring, ${\Bbb Z}_3$ or the product
of such rings. In light of~\cite[Corollary 3.2]{CS2},  $R$ is strongly weakly nil-clean.\end{proof}

Recall that a ring $R$ is WUU if for any unit $u\in R$, $1\pm u\in R$ is a unipotent. We now describe weakly exchange WUU ring and extend~\cite[Corollary 2.15]{D} as follows. 

\begin{thm} A ring $R$ is strongly weakly nil-clean if and only if\end{thm}
\begin{enumerate}
\item [(1)] {\it $R$ is weakly exchange;}
\vspace{-.5mm}
\item [(2)] {\it $R$ is WUU.}
\end{enumerate}
\begin{proof} $\Longrightarrow $ Clearly, every exchange ring is weakly exchange. Thus, this implication is obtained by~\cite[Corollary 2.15]{D}.

$\Longleftarrow$ Since $R$ is WUU, it is 2-UU. By virtue of Theorem 3.3, $R$ is strongly 2-nil-clean.
If $R$ has homomorphic image ${\Bbb Z}_3\times {\Bbb Z}_3$, then ${\Bbb Z}_3\times {\Bbb Z}_3$ is WUU.
But $(-1,1)\in {\Bbb Z}_3\times {\Bbb Z}_3$ is invertible, but $(-1,1)-(1,1)$ and $(-1,1)+(1,1)$ are not nilpotent, a contradiction. Therefore
$R$ is strongly weakly nil-clean, by Lemma 3.5.\end{proof}

\begin{cor} A ring $R$ is strongly weakly nil-clean if and only if\end{cor}
\begin{enumerate}
\item [(1)] {\it $R$ is weakly exchange;}
\vspace{-.5mm}
\item [(2)] {\it every unit in $R$ is strongly weakly nil-clean.}
\end{enumerate}
\begin{proof} $\Longrightarrow $ This is clear.

$\Longleftarrow$ Let $u\in U(R)$. Then there exists an idempotent $e\in R$ such that $w:=u\pm e\in N(R)$ and $ue=eu$. Hence, $e=w-u$ or $u-w$. Thus, $e\in U(R)$, and so $e=1$. This shows that $u\in \pm 1+N(R)$, ie.e, $R$ is WUU. This completes the proof, by Theorem 3.6.\end{proof}

The next observation is a generalization of~\cite[Theorem ??? and Corollary 2.15]{D}.

\begin{cor} Let $R$ be a ring. Then the following are equivalent:\end{cor}
\begin{enumerate}
\item [(1)] {\it $R$ is strongly nil-clean.}
\vspace{-.5mm}
\item [(2)] {\it $R$ is a weakly exchange UU ring.}
\vspace{-.5mm}
\item [(3)] {\it $R$ is a weakly exchange in which every unit is strongly nil-clean.}
\end{enumerate}
\begin{proof} $(1)\Rightarrow (3)$ This is obvious.

$(3)\Rightarrow (2)$ Let $u\in U(R)$. Then there exists an idempotent $e\in R$ and $w\in N(R)$ such that $u=e-w$. Hence, $e=u+w\in U(R)$. This implies that $e=1$, and so $u=1+w$. Thus, $R$ is UU, as desired.

$(2)\Rightarrow (1)$ In view of Theorem 3.6, $R$ is strongly weakly nil-clean rings. As $R$ is a UU ring, $-1\in 1+N(R)$, and so $2\in N(R)$.
This completes the proof by~\cite{CS2}.\end{proof}


\vskip10mm\begin{thebibliography}{99} \bibitem{AK} N. Arora and S. Kundu, Commutative feebly clean rings, {\it J. Algebra App.}, {\bf 16}(2017), 1750128 (14 pages), DOI: 10.1142/S0219498817501286.

\bibitem{CH} H. Chen, {\it Rings Related
Stable Range Conditions}, Series in Algebra 11, World Scientific,
Hackensack, NJ, 2011.

\bibitem{CS} H. Chen and M. Sheibani, Strongly 2-nil-clean rings, {\it J. Algebra Appl.},
{\bf 16}(2017) 1750178 (12 pages), DOI: 10.1142/S021949881750178X.

\bibitem{CS2} H. Chen and M. Sheibani, Strongly weakly nil-clean rings, {\it J. Algebra Appl.},
https://doi.org/10.1142/S0219498817502334.

\bibitem{Dan} P.V. Danchev, On weakly exchange rings, {\it J. Math. Univ. Tokushima}, {\bf 48}(2014), 17-¨C22.

\bibitem{D} P.V. Danchev, Weakly UU rings, {\it Tsukuba J. Math.}, {\bf 40}(2016), 101--118.

\bibitem{DA} P.V. Danchev, Weakly clean WUU rings, {\it Internat. J. Algebra}, {\bf 11}(2017), 9--14.

\bibitem{Da} P.V. Danchev and T.Y. Lam, Rings with unipotent units, {\it Publicat. Math. (Debrecen)},
{\bf 88}(2016), 449--466.

\bibitem{D}A.J. Diesl, Nil clean rings,
 {\it J. Algebra}, {\bf 383}(2013), 197--211.

\bibitem{L} T. Y. Lam, A First Course in Non-Commutative Rings, {\it Springer- Verlag, New York, Inc}, 1991.
\bibitem{KWZ} M.T. Kosan; Z. Wang and Y. Zhou, Nil-clean and strongly nil-clean rings, {\it J. Pure Appl. Algebra} (2015),
http://dx.doi.org/10.1016/j.jpaa.2015.07. 009.

\bibitem{SP} C. Selvaraj and S. Petchimuthu, On prime spectrums of 2- primal rings, {\it Bull. Inst. Math. Acad. Sinica (N. S.)}, {\bf 6}(2011), 73-84.
\bibitem{ST} A. Stancu, A note on commutative weakly nil clean rings,
 {\it J. Algebra Appl.}, {\bf 15}(2016); DOI:10.1142/S0219498816200012.

\bibitem{Y} Z.L. Ying; T. Kosan and Y. Zhou,
Rings in which every element is a sum of two tripotents, {\it Canad. Math. Bull.}, http://dx.doi.org/10.4153/CMB-2016-009-0.

\bibitem{Z} Y. Zhou, Rings in which elements are sum of nilpotents, idempotents and nilpotents, {\it J. Algebra App.}, {\bf 16}(2017), DOI: 10.1142/S0219498818500093.
\end{thebibliography}
\end{document}